\documentclass[letterpaper, 10 pt, conference]{ieeeconf}

\IEEEoverridecommandlockouts
\usepackage{cite}
\usepackage{amsmath,amssymb,amsfonts}
\usepackage{algorithmic}
\usepackage{graphicx}
\usepackage{textcomp}
\usepackage{bbm, xspace}
\usepackage{mathrsfs}
\usepackage{mathtools}
\usepackage{color}
\usepackage{epsfig}
\usepackage[acronym]{glossaries}
\usepackage{cancel}
\usepackage{empheq}
\let\labelindent\relax
\usepackage{enumitem}
\usepackage{textgreek}
\usepackage{stackengine}
\usepackage{dblfloatfix}
\usepackage[font=small]{caption}
\usepackage{empheq}
\usepackage{flushend}
\usepackage{subcaption}

\newcommand{\ie}{i.e.,\@\xspace} 
\newcommand{\eg}{e.g.,\@\xspace} 

\def\be{\begin{equation}}
	\def\ee{\end{equation}}

\newcommand{\bs}{\boldsymbol}
\newcommand{\mc}{\mathcal}

\definecolor{darkgreen}{rgb}{0.0, 0.5, 0.0}

\renewcommand{\emph}{\textit}

\makeglossaries
\newacronym{GNEP}{GNEP}{generalized Nash equilibrium problem}
\newacronym{NE}{NE}{Nash equilibrium}
\newacronym{NEP}{NEP}{Nash equilibrium problem}
\newacronym{GNE}{GNE}{generalized Nash equilibrium}
\newacronym{v-GNE}{v-GNE}{variational \gls{GNE}}
\newacronym{ISS}{ISS}{Input-to-state-stable}
\newacronym{PPA}{PPA}{proximal-point algorithm}
\newacronym{PPPA}{PPPA}{preconditioned \gls{PPA}}
\newacronym{VI}{VI}{variational inequality}
\newacronym{GAE}{GAE}{generalized aggregative equilibrium}
\newacronym{v-GAE}{v-GAE}{variational \gls{GAE}}
\newacronym{KKT}{KKT}{Karush--Kuhn--Tucker}
\newacronym{FQNE}{FQNE}{firmly quasinonexpansive}
\newacronym{FNE}{FNE}{firmly nonexpansive}
\newacronym{ADMM}{ADMM}{alternating direction method of multipliers}
\newacronym{MPMM}{MPMM}{modified proximal method of multipliers}
\newacronym{OPF}{OPF}{optimal power flow}
\newacronym{OPFP}{OPFP}{optimal power flow problem}
\newacronym{NUM}{NUM}{network utility maximization}
\newacronym{END}{END}{Estimation Network Design}
\newacronym{DGD}{DGD}{distributed gradient descent}
\newacronym{ST}{ST}{Steiner Tree}
\newacronym{FB}{FB}{forward-backward}

\newcommand{\0}{\bs 0}
\def\1{{\bs 1}}

\def\argmin{\mathop{\rm argmin}}

\newcommand{\col}{\mathrm{col}}
\newcommand{\avg}{\operatorname{avg}}

\def\diag{\mathop{\hbox{\rm diag}}}

\def\Null{\mathop{\hbox{\rm null}}}

\def\range{\mathop{\hbox{\rm range}}}

\def\grad#1{{\nabla_{\! #1}}}
\def\subd#1{{\partial_{\!#1}}}

\def\spose#1{\hbox to 0pt{#1\hss}}

\def\R{\mathbb{R}}
\def\nat{\mathbb{N}}

\DeclareSymbolFont{myletters}{OML}{ztmcm}{m}{it}
\DeclareMathSymbol{\uplambda}{\mathord}{myletters}{"15}

\def\QEDhereeqn{\eqno\let\eqno\relax\let\leqno\relax\let\veqno\relax\hbox{\QED}}
\def\QEDopenhereeqn{\eqno\let\eqno\relax\let\leqno\relax\let\veqno\relax\hbox{\QEDopen}}

\makeatletter \let\cl@part\relax \makeatother
\usepackage{nohyperref}
\usepackage[capitalize]{cleveref}

\def\k{{k \in \nat}}  

\crefname{thm}{Theorem}{Theorems}
\crefname{lem}{Lemma}{Lemmas}
\crefname{cor}{Corollary}{Corollaries}
\crefname{claim}{Claim}{Claims}
\crefname{axiom}{Axiom}{Axioms}
\crefname{conj}{Conjecture}{Conjectures}
\crefname{fact}{Fact}{Facts}
\crefname{hypo}{Hypothesis}{Hypotheses}
\crefname{assum}{Assumption}{Assumptions}
\crefname{prop}{Proposition}{Propositions}
\crefname{crit}{Criterion}{Criteria}
\crefname{standing}{Standing Assumption}{Standing Assumptions}
\crefname{defn}{Definition}{Definitions}
\crefname{exmp}{Example}{Examples}
\crefname{rem}{Remark}{Remarks}
\crefname{prob}{Problem}{Problems}
\crefname{prin}{Principle}{Principles}
\crefname{alg}{Algorithm}{Algorithms}
\crefname{figure}{Figure}{Figures}
\crefname{assumption}{Assumption}{Assumptions}

\crefname{thmlisti}{Theorem}{Theorems}
\crefname{lemlisti}{Lemma}{Lemma}
\crefname{asmlisti}{Assumption}{Assumption}

\newtheorem{theorem}{Theorem}
\newtheorem{corollary}{Corollary}
\newtheorem{lemma}{Lemma}

\newtheorem{assumption}{Assumption}
\newtheorem{proposition}{Proposition}

\newtheorem{standing}{Standing Assumption}

\newtheorem{example}{Example}

\newlist{thmlist}{enumerate}{1}
\setlist[thmlist]{label={\it(\roman{thmlisti})}, ref=\thepb{(\it \roman{thmlisti})},noitemsep,topsep=0em,leftmargin=*}

\newlist{thmlist2}{enumerate}{1}
\setlist[thmlist2]{label={(\alph{thmlist2i})}, ref={{(\alph{thmlist2i})}},topsep=0em}

\newlist{asmlist}{enumerate}{1}
\setlist[asmlist]{label={\it(\roman{asmlisti})}, ref=\theassumption{(\it \roman{asmlisti})},noitemsep,topsep=0em,leftmargin=*}

\usepackage{thmtools}

\Crefname{pb}{Problem}{Problems}
\addtotheorempostheadhook[pb]{\crefalias{thmlisti}{pb}}

\Crefname{asm}{Assumption}{Assumptions}
\addtotheorempostheadhook[assumption]{\crefalias{thmlisti}{assumption}}

\def\I{\mc{I}}
\def\i{i\in\I}

\def\P{\mc{P}}
\def\p{{p\in\mc{P}}}

\def\y{y}
\def\hy{\h{y}}
\def\hyt{\tilde{\h{y}}}
\def\hyu{\underline{\h{y}}}

\def\hyavg{\h{y}_{\avg}}

\def\hz{\h{z}}

\def\hzu{\underline{\h{z}}}
\def\hzbar{\bar{\h{z}}}
\def\hw{\h{w}}

\def\hg{\h{g}}

\def\hystar{\h{y}^{\star}}
\def\hzstar{\h{z}^{\star}}
\def\mcV {\mc{V}}
\def\hv{\h{v}}

\def\eigmin{\uplambda_{\operatorname{min}}}

\def \Rmc{{   \textrm{R}   }}

\def \fbs{{\bs{f}}}

\def\n#1{{n_{#1}}}

\def\Ebs#1{{\h{\mc{C}}}_{#1}}
\def\Ebsperp#1{{\h{\mc{C}}}{}_{\!\scriptscriptstyle \perp}^{  \mat{#1}}}

\def\Piparallel#1{\Pi_{\scriptscriptstyle \parallel}^{\mat{#1}}}
\def\Piperp#1{\Pi_{\!\scriptscriptstyle \perp}^{\mat{#1}}}

\def\mat{\mathrm}
\def\set{\mc}

\def\h{\bs}

\def\id{\mat{I}}
\def\Ntilde#1{ {N}_{#1}}

\def\E#1{\mc{E}^{\scriptstyle \text{#1}}}
\def\g#1{\mc{G}^{\scriptstyle \text{#1}}}
\def\V#1{\mc{V}^{\scriptstyle \text{#1}}}
\def\W#1{\mat W^{\scriptstyle \text{#1}}}

\def\w#1{w^{\scriptstyle \text{#1}}}

\def\neig#1#2{\mc{N}{~\!\!}^{\scriptstyle \text{#1}}({#2})}
\def\neigo#1#2{\overline{\mc{N}}{~\!\!}^{\scriptstyle \text{#1}}({#2})}

\def\gD#1{{\mc{G}}^{\scriptstyle \text{D}}_{#1}}

\def\WD#1{{\mat W}^{\scriptstyle \text{D}}_{#1}}

\def\ED#1{\mc{E}^{\scriptstyle \text{D}}_{#1}}
\def\neigD#1#2{{\mc{N}}{~\!\!}^{\scriptstyle \text{D}}_{#1}(#2)}

\def\gk#1#2{{\mc{G}}^{{\scriptstyle \text{#1}},#2}}
\def\Ek#1#2{\mc{E}^{{\scriptstyle \text{#1}},{#2}}}
\def\gknocomma#1#2{{\mc{G}}^{{\scriptstyle \text{#1}}#2}}
\def\Eknocomma#1#2{\mc{E}^{{\scriptstyle \text{#1}}{#2}}}
\def\gDk#1#2{{\mc{G}}^{{\scriptstyle \text{D}},#1}_{#2}}
\def\WDk#1#2{{\mat W}^{{\scriptstyle \text{D}},#1}_{#2}}
\def\EDk#1#2{\mc{E}^{{\scriptstyle \text{D}},{#1}}_{#2}}

\def\WDbs{ {\h{\mat W}}^{\scriptstyle \text{D}}}

\title{
	Estimation Network Design framework for efficient distributed optimization
}

\author{Mattia Bianchi and  Sergio Grammatico
	\thanks{M. Bianchi is with the Automatic Control Laboratory (IfA), ETH Zürich, Switzerland (\texttt{mbianch@ethz.ch}). S. Grammatico is with the Delft Center for Systems and Control (DCSC), TU Delft, The Netherlands (\texttt{s.grammatico@tudelft.nl}). This work is supported by NWO under research project OMEGA (613.001.702), by the ERC under research project COSMOS (802348) and by ETH Zürich funds.}
}

\begin{document}

\maketitle
\thispagestyle{empty}
\pagestyle{empty}

\begin{abstract}
    Distributed decision problems features a group of agents that can only communicate over a peer-to-peer network, without a central memory. In applications such as network control and data ranking, each agent is only affected by a small portion of the decision vector: this sparsity is typically ignored in distributed algorithms, while it could be leveraged to improve efficiency and scalability. To address this issue, our recent paper \cite{Bianchi_minG_TCNS_2023} introduces \gls{END}, a graph theoretical language for the analysis and design of distributed iterations.
    \gls{END} algorithms can be tuned  to exploit the sparsity of specific problem instances, reducing communication overhead and minimizing redundancy, yet without requiring case-by-case convergence analysis. In this paper, we showcase the flexility of \gls{END} in the context of distributed optimization. In particular, we study the sparsity-aware version of many established methods, including ADMM, AugDGM and Push-Sum DGD. Simulations on an estimation problem in sensor networks demonstrate that \gls{END} algorithms can boost convergence speed and greatly reduce  the communication and memory  cost.
\end{abstract}

%

\section{Introduction}\label{sec:introduction}
 Modern big data optimization problems in network control and  machine learning 
 are typically \emph{partially separable} \cite{NecoaraClipici_Coordinate_SIAM2016} -- i.e., the cost function is the sum of $N$ individual costs, each depending only on a small portion of the overall optimization variable. This structure is widely exploited in \emph{parallel} algorithms \cite{NecoaraClipici_Coordinate_SIAM2016,Richtarik_HMLR_2016} -- where multiple distinct processors share the computation cost, but having access to a common memory. Yet, this is not the case for \emph{distributed} scenarios -- where the processors (or agents) are constrained to communicate uniquely with some neighbors over a communication network. In fact, most distributed optimization methods entails the agents reaching consensus on the entire optimization vector \cite{Nedic_DIGing_SIAM2017,Scutari_Unified_TSP2021}, even when each agent eventually only uses a few components of the solution, as in resource allocation and network control \cite{Notarnicola_Partitioned_TCNS2018}. 
 This may result in prohibitive memory and communication requirements, and in poor scalability  if the decision vector grows with the network size. 

Efficient solutions are known for \emph{partitioned} problems, where the local cost functions (or constraints) only directly couple each agent to its  neighbors \cite{Notarnicola_Partitioned_TCNS2018,Erseghe_ADMM_SPL2012,Todescato_Partition_AUT2020,Giselsson_AUT2013,Dallanese_Microgrids_TSG2013}.
Notably, however, this requires that
the communication graph matches the \emph{interference} graph (describing the coupling among the agents in the cost or constraints),
which is usually not the case for wireless and ad-hoc networks. 

Remarkably, general partially-separable problems were addressed via distributed dual methods, by Mota et al. \cite{Mota:LocalDomains:TAC:2014} and
later by Alghunaim et al.  \cite{Alghunaim_SparseConstraints_TAC2020,Alghunaim_Stochastic_TAC2020}: 
in this approach, each component of the optimization variable  is estimated by a suitably chosen cluster of agents only. 
Unfortunately, the dual formulation is only effective over undirected communication networks.
Other works \cite{Rebeschini_Locality_TCNS2019,Brown_Locality_TCNS2021} rely on the concept of \emph{locality}, which result in improved efficiency, but at the cost of accuracy; further, any structure beyond distance on the communication graph is ignored. 

To deal with these challenges, in our recent work \cite{Bianchi_minG_TCNS_2023}, we  introduced \glsfirst{END}, a
graph-theoretic language to describe   how the estimates of any variable of interest (e.g., optimization vector, dual multipliers, cost gradient)  are allocated and combined among the agents in a generic distributed algorithm. \gls{END} allows assigning the estimate of each component of the variable of interest to a subset of the agents, 
according to the sparsity structure of a given problem --  without resorting to a case-by-case convergence analysis. Leveraging the problem sparsity is especially convenient in repeated or time-varying problems (e.g., distributed estimation and model predictive control (MPC) \cite{Mota:LocalDomains:TAC:2014}), where the one time-cost of efficiently assigning the estimates yields improved (iterative) online performance. 
Although \cite{Bianchi_minG_TCNS_2023}   focuses on distributed Nash equilibrium problems, the \gls{END} framework is flexible and applicable to virtually any distributed decision problem.

\emph{Contributions:} In this paper, we apply and tailor the END framework \cite{Bianchi_minG_TCNS_2023} to distributed optimization problems, thus  unifying and generalizing several recent approaches.  
For the case of dual algorithms, our setup retrieves the formulation in \cite{Mota:LocalDomains:TAC:2014,Alghunaim_Stochastic_TAC2020} (see Proposition~\ref{prop:dual_reformuation}). Here we present a novel sparsity-aware ADMM, but one can obtain the END version of virtually any dual algorithm (Section~\ref{sec:dualmethods}).  Further, compared to \cite{Mota:LocalDomains:TAC:2014,Alghunaim_Stochastic_TAC2020}, our framework has broader applicability:
\begin{itemize}
[leftmargin =1.5em]
    \item[(i)] \emph{it can be used for primal methods.} To illustrate, we present the END version of the ABC method \cite{Scutari_Unified_TSP2021},  encompassing many established algorithms.  As an example, we derive a gradient-tracking iteration where each agent only needs to estimate a fraction of the cost gradient -- the \gls{END} counterpart of AugDGM \cite{Xu_AugDGM_CDC2015} (Section~\ref{sec:ABC});
    \item[(ii)] \emph{it works on directed graphs.} Specifically, we present the sparsity-aware version of the Push-Sum DGD algorithm \cite{NedicOlshevsky_Directed_TAC2015}, that is guaranteed to converge over time-varying and column stochastic graphs (Section~\ref{sec:DGD}).
\end{itemize}
It will be clear from our arguments that,  thanks to our powerful stacked notation, the analysis of \gls{END} algorithms presents  little complication compared to their sparsity-unaware counterparts.  Nonetheless, the impact in terms of flexibility and performance is remarkable. We illustrate numerically this point on an estimation problem for wireless sensor network, where we observe that \gls{END} can decrease the communication cost by more than $90\%$ (Section~\ref{sec:numerics}).

\subsection{Background and notation}

\subsubsection{Basic notation}  $\mathbb{N}$ is the set of natural numbers, including $0$.
$\R$ ($\R_{\geq 0}$) is the set of (nonnegative) real numbers.
$\0_q\in \R^q$ ($\1_q\in\R^q$) is a vector with all elements equal to $0$ ($1$); $\mathrm{I}_q\in\R^{q\times q}$ is an identity matrix; the subscripts may be omitted when there is no ambiguity. $e_i$ denotes a vector of appropriate dimension with $i$-th element equal to 1 and all other elements equal to 0.
For  a matrix $ \mat A  \in \R^{p \times q}$, $[\mat A]_{i,j}$ is the element on  row $i$ and column $j$; $\Null(\mat A)\coloneqq \{x\in\R^q \mid \mat Ax=\0_n\}$ and $\range(\mat A)\coloneqq \{v\in\R^p \mid v=\mat Ax, x\in\R^q \}$.
If $ \mat A= \mat A^\top\in\R^{q\times q}$, $\uplambda_{\textnormal{min}}(\mat A)=:\uplambda_1(\mat A)\leq\dots\leq\uplambda_q(\mat A)=:\uplambda_{\textnormal{max}}(\mat A)$ denote its eigenvalues.
$\diag(\mat A_1,\dots,\mat A_N)$ is the block diagonal matrix with $\mat A_1,\dots,\mat A_N$ on its diagonal. Given $N$ vectors $ x_1, \ldots, x_N$,  $\col (x_1,\ldots,x_N ) \coloneqq  [ x_1^\top \ldots  x_N^\top ]^\top$. $\otimes$ denotes the Kronecker product.  Given a positive definite matrix $ \R^{q\times q} \ni \mat Q \succ 0$,  $\langle x \mid y \rangle _\mat{Q}=x^\top \mat Q y$ os the  $\mat{Q}$-weighted inner product,  $\| \cdot \|_{\mat Q}$ is the associated norm; we omit the subscripts if $\mat Q = \id $.
Given a function $\psi: \R^q \rightarrow \overline \R \coloneqq \R \cup \{\infty\}$, its set-valued subdifferential operator is denoted by
$\partial \psi: \R^q \rightrightarrows \R^q:x\mapsto  \{ v \in \R^q \mid \psi(z) \geq \psi(x) + \langle v \mid z-x \rangle , \forall  z \in \R^q \}$.

\subsubsection{Graph theory} \label{sec:graphtheory}
A (directed) graph $\g{}=(\mcV,\E{})$
consists of a nonempty set of vertices (or nodes) $\mcV=\{1,2,\dots, V \}$ and a set of edges $\E{}\subseteq \mcV \times \mcV $. We denote by $\neig{}{v}\coloneqq \{ u\mid (u,v) \in \E{} \} $ and $\neigo{}{v}\coloneqq \{ u \mid (v,u) \in \E{} \} $ the set of in-neighbors (or simply neighbors)  and out-neighbors of vertex $ v \in \mcV$, respectively.
A path  from $v_1\in \mcV $ to  $v_N\in \mcV$ of length $T$ is a sequence of vertices $(v_1,v_2,\dots,v_T)$ such that $(v_t,v_{t+1} )\in \E{}$ for all $t=1,\dots, T-1$. $\g{}$ is strongly connected if there exist a path from $u$ to $v$, for all $u,v \in \mcV{}$; in case $\g{}$ is undirected, namely if $(u,v) \in\E{}$ whenever $(v,u) \in\E{}$, we simply say that $\g{}$ is connected. The restriction of the  graph $\g{}$ to a set of vertices $\V{A} \subseteq \V{}$ is defined as $\g{}|_{\V{A}} \coloneqq ( \V{A}, \E{} \cap (\V{A} \times \V{A}))$.   We also write $\g{}=(\V{A}, \V{B},\E{})$ to highlight that $\mc{G}$ is bipartite, namely $\V{} = \V{A} \cup \V{B}$ and $\E{} \subseteq \V{A} \times \V{B}$.
We may associate to  $\g{}$ a weight matrix $\W{} \in \R^{V\times V}$ compliant with $\g{}$, namely $\w{}_{u,v}\coloneqq[\W{}]_{u,v}>0$ if $(v,u)\in \E{}$, $\w{}_{u,v}=0$ otherwise.
$\g{}$ is unweighted if $\w{}_{u,v}=1$ if $(v,u) \in \E{}$. 
Given two graphs $\g{A}=(\V{A},\E{A}) $ and $\g{B}=(\V{B},\E{B}) $, we write  $\g{A} \subseteq \g{B}$ if $\g{A}$ is a subgraph of $\g{B}$, \ie if $\V{A} \subseteq \V{B}$ and $\E{A} \subseteq \E{B}$; we define $\g{A} \bigcup \g{B} \coloneqq  (\V{A} \cup \V{B}, \E{A} \cup \E{B})$.
A
time-varying graph $(\gknocomma{}{k})_{\k}$, $\gknocomma{}{k}=(\V{},\Eknocomma{}{k})$ is $Q$-strongly connected if 
$\bigcup_{t=kQ}^{({k+1})Q-1} \gknocomma{}{t}$  is strongly connected for all $\k$. 
\section{\gls{END} for distributed optimization }\label{sec:mathbackground}

We first recall the general
\gls{END}
framework \cite{Bianchi_minG_TCNS_2023}, 
that describes the information structure in any distributed algorithm. It is characterized by:
\begin{enumerate}
[leftmargin=*]
	\item  a \textbf{set of agents} $\I \coloneqq\{ 1,2,\dots,N \}$;
	\item a given (directed) \textbf{communication graph} $\g{C}=(\I,\E{C})$, over which the agents can exchange information:  agent $i$ can receive data from agent $j$ if and only if
	$j\in\neig{C}{i}$;
	\item a \textbf{variable of interest} $y \in \R^ \n{y}$, partitioned as $y=\col((y_p)_{p\in \P})$,   $\P\coloneqq\{1,\dots,P\}$, $y_p\in\R^{\n{y_p}}$;
	\item a given  \textbf{interference graph} $\g{I}=(\P,\I,\E{I})$,  $\E{I}\subseteq \P \times \I$, specifying  which components of $y$  are indispensable for each agent:  $p\in \neig{I}{i}$ means that agent $i$ needs (an estimate of) $y_p$ to perform some essential local computation;\footnote{For ease of notation, assume $\neigo{I}{p}\neq \varnothing$ for all $\p$.}
    \item  a bipartite \textbf{estimate graph} $\g{E}=(\P,\I,\E{E})$, $\E{E}\subseteq \P \times \I$. Since agents might be unable to access $y$, each agent estimates some of the components $y_p$'s, as specified by the estimate graph: agent $i$ keeps an estimate $\h{y}_{i,p}\in\R^{\n{y_p}}$ of $y_p$ if and only if $p\in \neig{E}{i}$;
    \item $P$ directed \textbf{design  graphs}  $\{\gD{p} \}_{\p}$,  $\gD{p}=(\neigo{E}{p},\ED{p})$. $\gD{p}$  describes how the agents that estimate  $y_p$  exchange their estimates: agent $i$ can receive $\h{y}_{j,p}$ from agent $j$ if and only if $i\in\neigD{p}{j}$;
\end{enumerate}

%

\smallskip
\noindent Specifically, in this paper, we apply the \gls{END} framework to the distributed optimization problems
\begin{align} \label{eq:do}
     \min_{y\in\R^\n{y}} \  \textstyle\sum_{\i} f_i(y),
	\end{align}
 where $f_i:\R^\n{y} \rightarrow \overline \R$ is a private cost function of agent $i$. 
We choose the variable of interest in the \gls{END} framework to coincide with the optimization variable\footnote{Except for Section~\ref{subsec:constraintcoupled}, where we instead select the dual variable as variable of interest; see also \cite{Bianchi_minG_TCNS_2023} for different possible choices (e.g., aggregative values) in variational problems.}. Hence, we partition the optimization variable as $y=\col ((y_p)_{\p})$.

The common approach \cite{NedicOlshevsky_Directed_TAC2015,Scutari_NEXT_TSIPN2016} to  solve \eqref{eq:do} over a communication network $\g{C}$ is to  assign to each agent $\i$ a copy $\tilde{\hy}_i\coloneqq \col ((\h{y}_{i,p})_\p) \in \R^\n{y} $ of the whole decision variable 
	and to let the agents exchange their estimates with every neighbor over $\g{C}$; in \gls{END} notation, we write this as\footnote{In the following, we refer to \eqref{eq:standard} as the  ``standard'' choice, as it is the most widely studied scenario. With $\g{c} = \gD{p}$ we also imply $\W{C} = \WD{p}$.}
	\begin{align}	\label{eq:standard}
		\E{E} =\P \times \I, \qquad  
		\gD{p} =\g{C} \  (\forall \p).
		\end{align}
Yet, in  several applications, like network control and data ranking \cite{NecoaraClipici_Coordinate_SIAM2016},
each cost function $f_i$  depends only on some of the components of $y$, as specified by an interference graph $\g{I}$: $f_i$ depends on $y_p$ if and only if $p\in\neig{I}{i} \subseteq\P$. With some abuse of notation, we highlight this fact by writing 
\begin{align}\label{eq:fipartiallyseparable}
    f_i(y)=f_i ((y_p)_{p\in \neig{I}{i}}) .
\end{align}
Clearly, the standard choice \eqref{eq:standard} for the graphs $\g{E}$ and $\{\gD{p}\}_{\p}$ does not take advantage of the  structure in \eqref{eq:fipartiallyseparable}. In fact, agent $i$ only needs  $(y_p)_{p\in \neig{I}{i}}$ to evaluate (the gradient of) its local cost $f_i$; storing a copy of the whole vector $y$ could be unnecessary and  inefficient -- especially if  $\g{I}$ is sparse and  $\n{y}$ is large.

\subsection{Problem-dependent design,  unified analysis}\label{subsec:design}

From an algorithm design perspective, the graphs $\g{C}$ and $\g{I}$ shall be considered  fixed as part of the problem formulation. In contrast, the graphs $\g{E}$ and $\{\gD{p} \}_{\p}$ are design choices, although with some constraints.

\begin{standing}\label{asm:consistency}
	$ \g{E}$ and $\gD{p}$ are chosen such that $\g{I}\subseteq \g{E}$ and  $\gD{p}\subseteq \g{C}$  for all $\p$.  \hfill $\square$
\end{standing}

In particular,  $\g{E}\subseteq{\g{I}}$ means that each agent estimates at least the components of $y$ which are indispensable for local computation. Moreover, since the estimates are exchanged over $\gD{p}$ and communication can only happen over $\g{C}$, it must hold that $\gD{p}\subseteq \g{C}$. 
In addition, we will always need some level of connectedness for each graph $\gD{p}$, to ensure that the agents can reach consensus on the estimates of $y_p$, as for instance in the following condition. 

\begin{assumption}\label{asm:connected}
    For each $\p$, $\gD{p}$ is undirected and connected. \hfill $\square$
\end{assumption}

Designing $\g{E}$ and $\{\gD{p} \}_{\p}$ to satisfy \cref{asm:consistency} and \cref{asm:connected} is not difficult if $\g{C}$ is itself undirected and connected: one trivial choice is \eqref{eq:standard}. Yet, one wishes to also consider \emph{efficiency} specifications (e.g., in terms of memory allocation, communication or bandwidth) by imposing extra (soft) constraints on $\g{E}$ and $\{\gD{p} \}_{\p}$. We present a simple instance  in \cref{fig:0} and refer to \cite[App.~A]{Bianchi_minG_TCNS_2023} for  more examples. Such an optimal design is in general computationally expensive. Nonetheless, the performance advantages in terms of algorithm execution can be well worth the (one-time) cost of an efficient algorithm design, especially in 
repeated problems  \cite{Mota:LocalDomains:TAC:2014,Bianchi_Timevarying_LCSS2021} (where the same distributed problem is solved multiple times for different values of some parameters/measurements). 

Further, while this design procedure is very problem and goal dependent,  it does not affect the analysis of \gls{END} algorithms. By simply postulating some connectedness property, as in \cref{asm:connected}, we can unify the convergence analysis of standard algorithms (that use \eqref{eq:standard}) with that of methods specifically devised for problems with unique sparsity.

\begin{figure}[t]
\begin{subfigure}[a]{\columnwidth}
\centering
\includegraphics[width=0.8\columnwidth]
{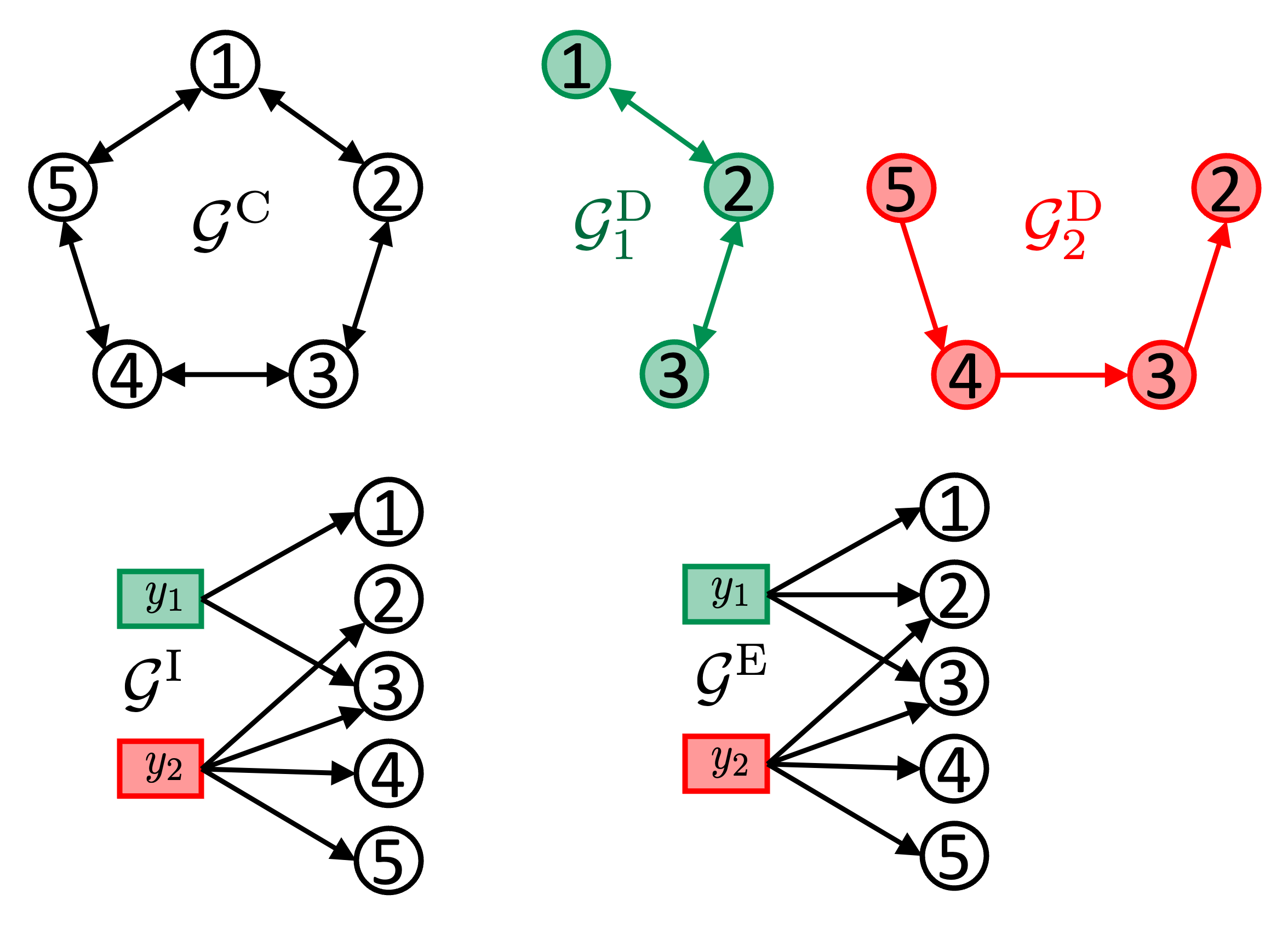}
\vspace{-2em}
\caption{}
\label{fig:0:A}\end{subfigure}
\begin{subfigure}[b]{\columnwidth}
\centering
\includegraphics[width=0.37\columnwidth]{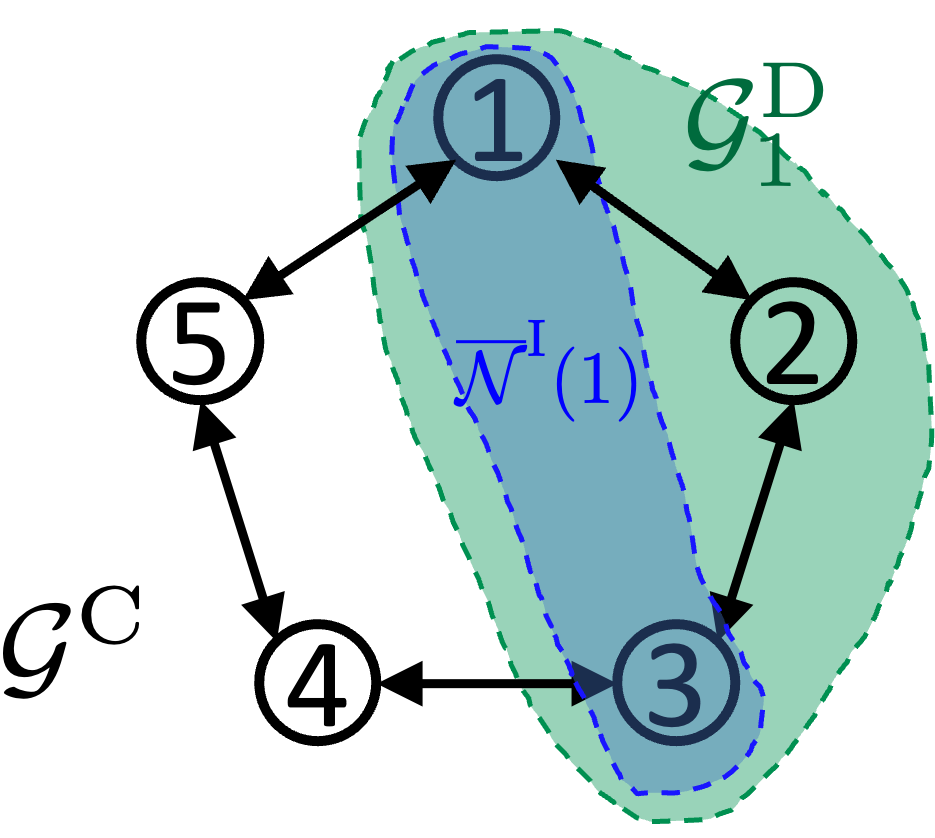}
\vspace{-1.5em}
\caption{}
\label{fig:0:B}
\end{subfigure}
\caption{
 (a) A simple example of \gls{END} design from \cite{Bianchi_minG_TCNS_2023}. On the left, the given communication and interference graphs, with $\mc{I} = \{1,2,3,4,5\}$ and $\mc{P} = \{1,2\}$. On the right a possible choice for the design graphs and the corresponding estimate graphs.
 \\
 (b) We focus on the design of $\gD{1}$. The given efficiency specification is to minimize the number of copies of $y_1$ (i.e., the number of nodes in $\gD{1}$), but provided that $\gD{1}$ is connected and \cref{asm:consistency} is met.
Note that agent $2$ has to estimate  $y_1$ (\ie $1 \in \neig{E}{2}$), even though agent $2$ is not directly affected by $y_1$ (\ie $1 \notin \neig{I}{2}$): otherwise, the information could not travel between nodes $1$ and $3$, which are not communication neighbors. In general, a solution to this design problem can be obtained by solving an Unweighted Steiner Tree problem \cite{Chalermsook:Distributed:Steiner:2005}, for which distributed off-the-shelf algorithms are available \cite{Chalermsook:Distributed:Steiner:2005}.
}\label{fig:0}
\end{figure}

\subsection{\gls{END} Notation}\label{subsec:notation}

We next introduce the stacked \gls{END} notation, crucial in our analysis.
For all $\p$, let $\Ntilde{p}\coloneqq\left| \neigo{E}{p} \right|$ be the number of copies of $y_p$. Recalling that that $\hy_{i,p}$ is the estimate of   $\y_p$ kept by agent $i$, we define:
\begin{align}
	\hy_p & \coloneqq \col ( ( \h{y}_{i,p} )_{i\in\neigo{E}{p}} )\in \R^{  \Ntilde{p} n_{y_p}}, \quad \forall \p;
	\\
	\hy   & \coloneqq\col( ( {\h{y}}_p )_{\p} ) \in \R^\n{\hy},
\end{align}
where   $\n{\hy}\coloneqq\sum_{\p} \Ntilde{p} \n{y_p}$.  Note that $\hy_p$ collects all the copies of $y_p$ kept by different agents. We denote
\begin{align}\label{eq:WDbs}
    \WDbs \coloneqq \diag ( ( \WD{p}\otimes \textrm{I}_{\n{y{_p}}})_{\p} ),
\end{align}
where $ \WD{p}$ is the weight matrix of $\gD{p}$.  Let
\begin{align}\label{eq:consensusp}
	\Ebs{p} &\coloneqq\{\hy_p \in \R^{N_p \n{y_p}} \mid \hy_p = \1_{N_p} \otimes v, v\in\R^{\n{y_p}} \},
\end{align}
be the consensus space for $\hy_p $ (where all the estimates of $\y_p$ are equal);   $\Ebs{} \coloneqq \textstyle \prod_{\p} \Ebs{p}$ be the overall consensus space;
\begin{align}\label{eq:2consensus}
    \Ebs{}(\y)\coloneqq\col(( \1_{N_p} \otimes \y_p )_{\p} ).
\end{align}
For each $\p$, for each $i\in \neigo{E}{p}$, we denote by \begin{align}\label{eq:ip}
    i_p\coloneqq \textstyle \sum_{j\in \neigo{E}{p}, \, j\leq i} 1
\end{align} 
the  position of $i$ in the ordered set of nodes  $\neigo{E}{p}$. For all $i \in \neig{E}{p}$, we denote by $\Rmc_{i,p}\in \R^{\n{y_p} \times {N_p \n{y_p}}}$ the matrix that selects $\hy_{i,p}$ from $\hy_p$, \ie  $\hy_{i,p}= \Rmc_{i,p} \hy_p$.

Sometimes it is useful to define agent-wise quantities, indicated with a tilde. Let 
\begin{align}
	\hyt_i \coloneqq  \col ( ( \hy_{i,p} )_{p\in \neig{E}{i}} ), \
	\hyt \coloneqq\col( ( \hyt_i )_{\i} ) \in\R^{\n{\hy}},
\end{align}
where $\hyt_i$ collects all the estimates kept by agent $i$.

\section{Distributed optimization algorithms}\label{sec:do_algorithms}

In this section,  we leverage the \gls{END} framework to extend several distributed optimization algorithms by exploiting partial coupling. We recall the cost-coupled problem in \eqref{eq:do}:
\begin{align}
\label{eq:do1}
\min_{ y\in\R^{\n{y}}} \ {f}(y) \coloneqq \textstyle\sum_{\i} f_i( (y_p)_{p\in\neig{I}{i}} ),
\end{align}
where $f_i$ is a private cost function of agent $i$, and  we choose the optimization variable $y = \col( (y_p)_{\p})$ as the variable of interest in the \gls{END}; with some usual overloading, we write
\begin{align} 
f_i(y) = f_i(  (y_p)_{p\in \neig{E}{i}}  ) = f_i ( (y_p)_{p\in \neig{I}{i}}  ).
\end{align}
Let $\set Y^\star $ be the solution set of \eqref{eq:do1}, assumed to be nonempty.

\subsection{\gls{END} dual methods} \label{sec:dualmethods}
Under \cref{asm:consistency}, we  can recast \eqref{eq:do1} by introducing local estimates and consensus constraints.The   following redormulation is not novel, and in fact it was employed for the dual methods in \cite{Mota:LocalDomains:TAC:2014,Alghunaim_SparseConstraints_TAC2020,Alghunaim_Stochastic_TAC2020}. 
\begin{proposition}\label{prop:dual_reformuation}
	Let \cref{asm:connected} hold. Then, problem \eqref{eq:do1} is equivalent to:
		\begin{equation}\label{eq:dual_reformulation}
		\left\{
		\begin{aligned}
			\min_{\hyt \in\R^\n{\hy}}&  \  \textstyle\sum_{\i} f_i( \hyt_i ) = f_i((\hy_{i,p})_{p\in \neig{E}{i}})
			\\
			\text{s.t.} & \  \hy_{i,p}=\hy_{j,p}  \quad  \forall \p,\forall (i,j)\in \ED{p}. 
			\end{aligned} \right.
			\\[-1.1em]
		\end{equation}
	\hfill$\square$
\end{proposition}

If $\gD{p}=\g{C}$ for all $\p$, then \eqref{eq:dual_reformulation} reverts to the formulation used in standard dual methods \cite{Boyd_ADMM_2010,Bastianello_ADMM_TAC2021,Uribe:etal:AdualApproach:OMS:2021}: these algorithms require each agent to store a copy of the whole optimization vector.  Instead, choosing a sparse $\g{E}$   can conveniently reduces the number  constraints in  \eqref{eq:dual_reformulation}.
Regardless, due to its structure (\ie separable costs and coupling constraints compliant with $\gD{p}$, hence with the communication graph), the problem in  \eqref{eq:dual_reformulation} can be immediately  solved via several  established Lagrangian-based algorithms (provided that the functions $f_i$'s are sufficiently well-behaved). In practice, this allows one to extend most (virtually all) the existing dual methods to the \gls{END} framework. 

\begin{example}[\gls{END} ADMM]
	Let  \cref{asm:connected} hold, and assume that $f_i$ is proper closed convex, for all $\i$. Applying the \gls{ADMM}  in \cite{Bastianello_ADMM_TAC2021} to  \eqref{eq:dual_reformulation}\footnote{After decoupling the constraints in  \eqref{eq:dual_reformulation} by introducing auxiliary bridge variables as $\{ \hy_{i,p}=h_{(i,j),p}, h_{(i,j),p}=h_{(j,i),p}, h_{(j,i),p}=\hy_{j,p} \}$;  the approach is standard and we refer to \cite{Bastianello_ADMM_TAC2021} for a complete derivation.} results in the iteration
	\begin{subequations}
	 \label{eq:ENDADMM}
		\begin{align}
			 \hyt_{i} ^{k+1} & =\underset{\hyt_{i}}{\argmin} \,
			\begin{aligned}[t]
			\Bigl\{
			f_i (\hyt_{i})   +  \textstyle\sum_{p\in\neig{E}{i}}\textstyle \sum_{j\in \neigD{p}{i}} \Bigl(
			\| \hy_{i,p} \|^2 
			\label{eq:ENDADMM:a}
	     	\\
			-  \langle z_{i,j,p},   \hy_{i,p} \rangle \Bigr) \Bigr\} \end{aligned}
			\\
			z_{i,j,p}^{k+1} &=  (1-\alpha )z_{i,j,p}^{k} -\alpha z_{j,i,p}^{k} +2\alpha \hy_{j,p}^{k+1},
			\label{eq:ENDADMM:b}
		\end{align}
		\end{subequations}
	where $z_{i,j,p}$  is an auxiliary variable kept by agent $i$,  for each $\i$, $p\in \neig{E}{i}$, $j\in \neigD{p}{i}$. Then, for any $\alpha \in (0,1)$,  $\hy_{i,p}$ converges to $y_p^\star$, where $y^\star=\col((y_p^\star)_{\p})$ is a solution of \eqref{eq:do1}, for all $\i$ and $p\in\neig{E}{i}$.
	Note that performing the update \eqref{eq:ENDADMM:b} requires agent $i$ to receive data from its neighbor $j \in \neigD{p}{i}$ (while \eqref{eq:ENDADMM:b} requires no communication). 
	If $\gD{p}=\g{C}$ for all $\p$, then the method retrieves the standard \gls{ADMM} for consensus optimization \cite[Eq.~(13)]{Bastianello_ADMM_TAC2021}. Yet, in general \eqref{eq:ENDADMM}  requires the agents to store and exchange  less (auxiliary) variables.
	\hfill $\square$
\end{example}

While 
\cref{prop:dual_reformuation} would hold even if the graphs $\gD{p}$'s are only strongly connected, distributed algorithms to efficiently solve \eqref{eq:dual_reformulation} typically require undirected communication. 


\subsection{\gls{END} ABC  algorithm}
\label{sec:ABC}
In this subsection, we propose an \gls{END} version of the ABC algorithm, recently developed in \cite{Scutari_Unified_TSP2021}. For  differentiable costs $f_i$'s, let us consider the iteration: $(\forall \i)(\forall p \in \neig{E}{i})$
\begin{subequations}\label{eq:ENDABC_i}
	\begin{align}
		\label{eq:ENDABC_i:a}
		\hy_{i,p}^{k+1} &= - \hz_{i,p}^{k}+
		 \sum_{j \in \neigo{E}{p} } [\mat A_p]_{i_p,j_p} \hy_{j,p}^k - \gamma  [ \mat B_p]_{i_p,j_p} \grad{y_p}  f_j(\hyt_{j}^k)
		\\
		\label{eq:ENDABC_i:b}
		\hz_{i,p}^{k+1} & = \hz_{i,p}^{k} +    \sum_{j \in \neigo{E}{p} } [\mat C_p]_{i_p,j_p}  \hy_{j,p}^{k+1},
	\end{align}
\end{subequations}
where $\hz_{i,p} \in \R^{\n{y_p}}$ is a local variable kept by agent $i$; for all $\p$,  $\mat A_p, \mat  B_p, \mat C_p$ are matrices in $ \R^{N_p \times N_p}$; $\gamma>0$ is a step size; and we recall the notation in \eqref{eq:ip}. Note that if the matrices $\mat A_p, \mat B_p, \mat C_p$'s are compliant with the corresponding graphs $\gD{p}$'s (\eg $\mat A_p = \mat B_p = \mat C_p = \WD{p}$), then the iteration \eqref{eq:ENDABC_i} is distributed. We can rewrite \eqref{eq:ENDABC_i}  in stacked form as
\begin{subequations}\label{eq:ENDABC}
	\begin{align}
		\label{eq:ENDABC:a}
		\hy^{k+1} &=   \mat A	\hy^{k} - \gamma \mat B \grad{\hy} \fbs  (\hy^{k}) - \hz^{k}
		\\
		\label{eq:ENDABC:b}
		\hz^{k+1} & = \hz^{k} +  \mat C  \hy^{k+1},
	\end{align}
\end{subequations}
where $ \mat A\coloneqq \diag (( \mat A_p\otimes \id_{\n{y_p}}  )_{\p} )$, $ \mat B\coloneqq \diag (( \mat B_p\otimes \id_{\n{y_p}}  )_{\p}  )$,  $ \mat C\coloneqq \diag (( \mat C_p\otimes \id _{\n{y_p}}  )_{\p} )$ belong to $\R^{\n{\hy}\times \n{\hy}}$, $\hz \coloneqq  \col ( ( \hz_p )_{\p} )$ with $\hz_p \coloneqq \col ((\hz_{i,p} )_{i\in \neigo{E}{p}})$, and $\fbs(\hy)\coloneqq\sum_{\i} f_i (\hyt_i)$. If  $\gD{p} = \g{C}$ for all $p$, and $\mat A_p$, $\mat B_p$, $\mat C_p$ are independent of $p$, then \eqref{eq:ENDABC} retrieves the ABC algorithm \cite[Eq.~3]{Scutari_Unified_TSP2021}.

We next charachterize the asymptotic behavior of \eqref{eq:ENDABC} for appropriately chosen $\mat A$, $\mat B$, $\mat C$ (all the proofs are in appendix). We recall the notation in \eqref{eq:consensusp}-\eqref{eq:2consensus}.
\begin{theorem}\label{th:ENDABC}
    Let $\mat D\coloneqq \diag( (\mat D_p\otimes \id_{\n{y_p
	}} )_{\p} )$, for some  $\{ \mat D_p \in \R^{N_p \times N_p} \}_{ \p }$. Assume that   $f_i$ is $L$-smooth and convex for each $\i$, and that:
	\begin{thmlist2}
		\item 
		\label{C1:th:ENDABC}  $\mat A=\mat B \mat D$ and $\mat B\succcurlyeq 0$, $\mat D\succ 0$;
		\item
		\label{C2:th:ENDABC} $(\forall \hy\in \Ebs{}) $ $\mat D \hy = \hy$, $ \mat B \hy = \hy$;
		\item
		\label{C3:th:ENDABC} $\mat C\succcurlyeq 0$,  $\Null(\mat C)=\Ebs{}$;
		\item 
		\label{C4:th:ENDABC} $\mat B$ and $\mat C$ commute: $\mat B \mat C =\mat C \mat B$;
		\item
		\label{C5:th:ENDABC} $\id -\frac{1}{2} \mat C -\sqrt{\mat B} \mat D \sqrt{\mat B} \succcurlyeq 0$.
	\end{thmlist2}
	Let $y^\star \in \set Y ^ \star$, $\hy^\star \coloneqq \Ebs{} (y^\star) $,
	and consider the  merit function
	\begin{equation}
		\mathfrak M (\hy)\coloneqq \max \{ \| \Piperp{} \hy \| \|\grad{\hy} \fbs( \hy^\star ) \|, | \fbs(\hy ) - \fbs(\hy^\star) |  \}.
	\end{equation}
	Then, for any $\hy^0\in\R^{\n{\hy}}$, $\hz^0 = \0_{\n{\hy}}$, $ \gamma \in  (\frac{\eigmin (D) }{L})$, the sequence $(\hy^k)_\k$ generated by \eqref{eq:ENDABC} satisfies
	\begin{equation}\label{eq:avg_conv_ABC}
		\textstyle  \mathfrak M \left( \hyavg^k \right) \leq \mc{O} (\frac{1}{k}),
	\end{equation}
	for all $\k$, where $\hyavg^k \coloneqq \frac{1}{k} \sum_{t=1} ^k \hy^t$.
	\hfill $\square$
\end{theorem}

It is shown in \cite{Scutari_Unified_TSP2021} that many celebrated schemes for consensus optimization can be retrieved as particular instances of the ABC algorithm, by suitably choosing  the matrices $\mat A$, $\mat B$, $\mat C$  \cite[Tab.~2]{Scutari_Unified_TSP2021}: 
EXTRA \cite{Shi_EXTRA_SIAM2015},
  NEXT \cite{Scutari_NEXT_TSIPN2016},
 DIGing \cite{Nedic_DIGing_SIAM2017},
 NIDS \cite{LiEtal_NIDS_TSP2019},
and others. 
 \cref{th:ENDABC} allows the extension of each of these methods to the  \gls{END} framework. We only discuss an example below; for the other schemes, the analysis can be carried out analogously, see also \cite[§III.A]{Scutari_Unified_TSP2021}.

\begin{example}[\gls{END} AugDGM]
	The following gradient-tracking algorithm is the \gls{END} version of \cite[Alg.~1]{Xu_AugDGM_CDC2015}: $(\forall \i)(p\in \neig{E}{i})$
	\begin{subequations}
			\begin{align*}
				\hy_{i,p}^{k+1} &= \hspace{-0.7em}  \sum_{j \in \neigD{p }{i}}
				[\WD{p}]_{i_p,j_p}  ( \hy_{j,p}^k-\gamma \hv_{j,p}^k)
				\\
	    	\hv_{i,p}^{k+1} & = \hspace{-0.7em} \sum_{j \in \neigD{p }{i}} \hspace{-0.5em}
				[\WD{p}]_{i_p,j_p}  (\hv_{j,p}^{k} +   \grad{y_p} f_j (\hyt^{k+1}) -  \grad{y_p} f_j (\hyt^{k})),
			\end{align*}
	\end{subequations}
	or, in stacked form,
		\begin{subequations} \label{eq:NEXT_ABC}
		\begin{align}
		\hy^{k+1} &=  \WDbs (\hy^k - \gamma \hv^k)
		\\
		\hv^{k+1} & =  \WDbs (\hv^{k} +   \grad{\hy} \fbs (\hy^{k+1}) -  \grad{\hy} \fbs (\hy^{k}) );
		\end{align}
	\end{subequations}
	we impose $\hy(0) = \0$, $\hv(0) = \WDbs{} \grad{\hy} \fbs (\hy^0)$. Here, $\hv_{i,p}$ represents an estimate of $\grad{y_p} \sum_{j \in \mc{I}} f_j (y)/N_p$ kept by agent $i$. Note that  agent $i$ only estimates and exchanges the components of the cost gradient (and of the optimization variable) specified by $\neig{E}{i}$, instead of the whole vector as in \cite[Alg.~1]{Xu_AugDGM_CDC2015} -- the two algorithms coincide only if  $\WD{p}=\W{C}$ for all $\p$.
	By eliminating the $\hv$ variable in \eqref{eq:NEXT_ABC}, we obtain 
	\begin{align}\label{eq:ABC_NEXT_elz}
		\begin{aligned}
			\hy^{k+2} &=  2\WDbs{} \hy^{k+1} -  {(\WDbs{} )}^2  \hy^{k}
			\\
			& \hphantom{{}={} } 
			- \gamma  {(\WDbs{} )}^2  
			(\grad{\hy} \fbs (\hy^{k+1}) -  \grad{\hy} \fbs (\hy^{k})) .
		\end{aligned}
	\end{align}
	Instead, eliminating $\hz$ from \eqref{eq:ENDABC} we get
	\begin{align}
		\begin{aligned}
			\hy^{k+2} &=  (\id-\mat C+\mat A ) \hy^{k+1} - \mat A  \hy^{k}
			\\
			& \hphantom{{}={} } - \gamma\mat B  (\grad{\hy} \fbs (\hy^{k+1}) -  \grad{\hy} \fbs (\hy^{k})) .
		\end{aligned}
	\end{align}
	which retrieves  \eqref{eq:ABC_NEXT_elz}  for $\mat A=\mat B={(\WDbs)
	}^2$, $\mat C={(\id- \WDbs)}^2$.\footnote{In fact, the sequence $(\hy^k)$ generated by \eqref{eq:ENDABC} coincide with that generated by \eqref{eq:NEXT_ABC}
    for the given initialization.} This choice satisfies the conditions in \cref{th:ENDABC}, with $\mat D=\id$, under  \cref{asm:connected} and doubly stochasticity.\footnote{Note that the properties of $\WD{p}$'s easily translate to $\WDbs$ due to the block structure in \eqref{eq:WDbs}. For instance, under the stated conditions, clearly $\Null(I -\WD{p})^2 = \range(\1_{N_p}) $ and $\Null(I -\WDbs)^2 = \Ebs{}$. } 
	\begin{corollary}
		Let \cref{asm:connected} hold; assume that $\WD{p} \1_{N_p}=\1_{N_p}$, $\WD{p}={\WD{p} }^\top$, for all $ \p$, and that $f_i$ is $L$-smooth and convex, for all $\i$. Then, for any $\gamma \in (0,\frac{1}{L})$ the rate \eqref{eq:avg_conv_ABC} holds for \eqref{eq:NEXT_ABC}.  \hfill $\square$
	\end{corollary}
\end{example}

\cref{th:ENDABC} requires a recovery procedure (i.e., \eqref{eq:avg_conv_ABC} holds for the running average only), as e.g. in \cite{Qu_Harnessing_TCNS2018}, but pointwise convergence could be shown for several special cases of \eqref{eq:ENDABC}, see e.g. \cite{Xu_AugDGM_CDC2015}. 
We  note  that  \cref{th:ENDABC} enhances customizability with respect to  \cite[Th.~24]{Scutari_Unified_TSP2021},   even in the standard scenario \eqref{eq:standard} (the sparsity-oblivious case), by allowing for non-identical blocks $\mat A_p$'s, $\mat B_p$'s, $ \mat C_p$'s -- corresponding to integrating different methods for the components of $y$.


\subsection{\gls{END} Push-sum DGD}
\label{sec:DGD}
Techniques to solve optimization problems over switching or directed graphs also find their counterpart in the \gls{END} framework. As an example, here we generalize the push-sum subgradient algorithm in \cite[Eq.~(1)]{NedicOlshevsky_Directed_TAC2015}.

Let the agents  communicate over a time-varying network $(\gk{C}{k})_{\k}$, $\gk{C}{k} =(\I,\Ek{C}{k})$. Given a fixed estimate graph $\g{E} \supseteq \g{I}$, for each $\p$ we consider a time-dependent design graph $(\gDk{k}{p})_{\k}$,  $\gDk{k}{p} = (\neigo{E}{p},\EDk{k}{p}) \subseteq \gk{C}{k}$ (note that the set of nodes is fixed in $\gDk{k}{p} $). 
For all $\i$ and  $ p \in \neig{E}{i}$, agent $i$ performs the following updates:
\begin{subequations}\label{eq:DGD_END}
	\begin{align}
		q_{i,p}^{k+1} & = \textstyle \sum _{j \in \neigo{E}{p} }  [ \WDk{k}{p} ]_{i_p,j_p} q_{j,p} ^k
		\\
		\hw_{i,p}^{k+1} & \coloneqq \textstyle \sum _{j \in \neigo{E}{p} }  [ \WDk{k}{p} ]_{i_p,j_p} \hz_{j,p} ^k
		\\
		\hg_{i,p} ^{k+1} &\in \subd{y_p}{f_i ( \hyt_{i}^{k+1}) },  \quad   \hy_{i,p}^{k+1}\coloneqq \frac {\hw_{i,p}^{k+1} } {q_{i,p}^{k+1} }
		\\
		\hz_{i,p}^{k+1} & = \hw_{i,p}^{k+1} - \gamma^{k} \hg_{i,p} ^{k+1},
	\end{align}
\end{subequations}
initialized at $\hz_{i,p}^{0} \in \R^{n_{y_p}}$, $q_{i,p}^0=1$. 
With respect to \cite[Eq.~(1)]{NedicOlshevsky_Directed_TAC2015},  agent $i$ keeps  one scalar  $q_{i,p}$  for each $p\in \neig{E}{i}$ (instead of one overall), but does not store and exchange  the variables $\hz_{i,p} \in \R^{\n{y_p}}$  for $p \notin \neig{E}{i}$.

\begin{assumption}\label{asm:columnstoch} 
   For all $\k$ and $\p$, it holds that:
\begin{itemize}
	\item[(i)] {\it Self-loops:}  for all $ i\in \neigo{E}{p}$, $(i,i) \in \EDk{k}{p}$;
	\item[(ii)] {\it Column-stochasticity}:  $\1_{N_p}^\top \WDk{k}{p} = \1_{N_p}^\top$;
	\item[(iii)] {\it Finite weights}:  $[\WDk{k}{p}]_{i_p,j_p} \geq \nu >0, \ \forall  (i,j) \in \EDk{k}{p}$.
	 \hfill $\square$
\end{itemize}
\end{assumption}
\begin{assumption}\label{asm:Qconnected}
    There exists an integer $Q>0$ such that, for all $\p$, $(\gDk{k}{p})_{\k}$ is $Q$-strongly connected.  \hfill $\square$
\end{assumption}

\begin{example}[Choosing time-varying design graphs] \sloppy ~Assume 
	$\g{C} \subseteq \bigcup_{t=kQ}^{(k+1)Q-1} \gk{C}{k}$, for all $\k$ and some strongly connected graph $\g{C}$.
	Choose some graphs $(\gD{p})_{\p}$ that satisfy \cref{asm:consistency} and such that each $\gD{p}$ is strongly connected. Then, \cref{asm:Qconnected}  holds by settihg $\gDk{k}{p}=\gD{p} \bigcap  \gk{C}{k}$, for all $\p$  and all $\k$. \hfill $\square$
\end{example}

\begin{theorem}\label{th:DGDABC}
 Let \cref{asm:columnstoch,asm:Qconnected} hold. Assume that, for all $\i$, $f_i$ is  convex and there is $L>0$ such that $\| g_i \|\leq L$, for all $y \in \R^\n{y}$ and  $g_i \in \subd{y} f_i (y) $. Let $(\gamma^k)_{\k}$ be a positive non-increasing sequence such that $\sum_{k=0}^{\infty} \gamma^k =\infty$, $\sum_{k=0}^{\infty} (\gamma^k)^2<\infty$. Then, the sequence $(\hy^k)_\k$ generated by \eqref{eq:DGD_END} converges to  $\Ebs{}(y^\star)$, for some $y^\star \in \mc{Y}^\star$. \hfill $\square$
\end{theorem}

\subsection{Constraint-coupled distributed optimization}\label{subsec:constraintcoupled}

Finally, we study  a different,  constraint-coupled problem:
\begin{subequations}\label{eq:CCP}
	\begin{empheq}[left={ \hspace{-1em}\empheqlbrace }]{align}
		\min_{ x_i \in\R^ {\n{x_i}}, \i} &  \  \textstyle \sum_{\i} f_i(  x_i ) \label{eq:CCP:a}
		\\
		\text{s.t. \quad   } & \  \textstyle \sum_{i \in \neigo{I}{p} } \mat A_{p ,  i}  x_i  - a_{p, i} = \0, \quad \forall \p \label{eq:CCP:b}
	\end{empheq} 
\end{subequations}
for a given interference graph $\g{I} = (\set P, \set I,\E{I})$, where $f_i$ and $ \{ \mat A_{p ,  i} \in \R^{\n{y_p} \times \n{x_i}},  a_{p,i} \in \R^{\n{y_p}} \}_{  p \in \neig{I}{i}  }$ are private data kept by agent $i$; and the constraints \eqref{eq:CCP:b} are  \emph{not} compliant with the communication graph $\g{C}$, namely $\neigo{I}{p} \not\subseteq \neig{C}{i}$ for any $i$. Differently from what we did with the cost-coupled problem in \eqref{eq:do1}, here we choose as the variable of interest the dual variable associated with the constraints in \eqref{eq:CCP:b}, \ $y = \col((y_p)_\p)\in\R^\n{y}$.
Typical distributed methods to solve \eqref{eq:CCP} require each agent to store a copy of the entire dual variable  (and possibly of  other variables in $\R^\n{y}$, e.g., an estimate of the constraint violation) \cite{Falsone_TrackingADMM_AUT2020,Li:CoupledInequality:TAC:2021}. 
\gls{END} primal-dual or dual methods can improve efficiency by exploiting the sparsity of $\g{I}$.
 For instance, (a simplified version of) the algorithm in \cite[Eq.~(31)]{Bianchi_minG_TCNS_2023} can be directly used to solve \eqref{eq:CCP}. 
 Alternatively, let us consider the dual of  \eqref{eq:CCP}:
\begin{align} \label{eq:dualCCP}
	\max_{ y \in\R^ {\n{y}} }  &  \  \textstyle \sum_{\i} \varphi_i (  (y_p)_{p\in \neig{I}{i}} ),
\end{align}
 $\varphi_i (y) \coloneqq \min_{x_i \in \R^{\n{x_i}} }  f_i(x_i) + \sum_{ p \in \neig{I}{i}} \langle y_p  , \mat  A_{p,i} x_i -a_{p,i} \rangle$; note that \eqref{eq:dualCCP} is  in the form \eqref{eq:do1}. In fact, \eqref{eq:dualCCP} was solved in  \cite{Alghunaim_SparseConstraints_TAC2020} via the reformulation \eqref{eq:dual_reformulation}; this approach has the disadvantage of requiring undirected communication. Nonetheless, \eqref{eq:dualCCP} can also be solved over directed (time-varying) networks, \eg via the iteration in \eqref{eq:DGD_END}.\footnote{If each $f_i$ is convex with compact domain, where the subgradients of the local dual function $\varphi_i$ can be computed as $\hg_{i,p}^k = \mat A_{p,i}x_i^\star(\hyt^k_i) -a_{p,i}$, with $x_i^\star (\hyt_i) \in \argmin_{x_i \in \R^{\n{x_i}} }  f_i(x_i) + \sum_{ p \in \neig{I}{i}} \langle \hy_{i,p}  ,  \mat A_{p,i} x_i -a_{p,i} \rangle  $.}


\section{Illustrative application}\label{sec:numerics}
In this section we study numerically a  regression problem with sparse measurements \cite{NecoaraClipici_Coordinate_SIAM2016,Alghunaim_Stochastic_TAC2020}, arising from distributed estimation in wireless and ad-hoc sensor networks. 
Let us consider some sensors $ \{1,2,\dots,N\}\eqqcolon \set I$ and some sources $\{1,2,\dots, P\} \eqqcolon \set P$, spatially distributed on a plane in the square $[0,1] \times [0,1]$, as illustrated in \cref{fig:distribution}. Each source $p$ emits a signal $\bar y_p \in \R$, sensed by all the sensors in a radius $r_{\textnormal{s}} >0$; in turn, each sensor $i$ measures
\begin{align}
    h_i \coloneqq   \mat H_i \; \col ((\bar y_p)_{p \in \neig{I}{i}}) + w_i,
\end{align}
 where $h_i \in \R^{\n{h_i}}$, $\mat H_i$ is a known output matrix, $w_i$ is the measurement noise.  Sensor $i$ can send information to all the peers in a radius $r_\textnormal{c}^i$ (e.g., proportional to the sensor specific power); this induces a directed communication network $\g{C} = (\set I, \E{C})$ among the sensors, which we assume to be \emph{strongly connected}.

\subsubsection{Linear regression}\label{sec:linearregression}
\begin{figure}[t] 
\centering
\includegraphics[width=0.4\columnwidth]{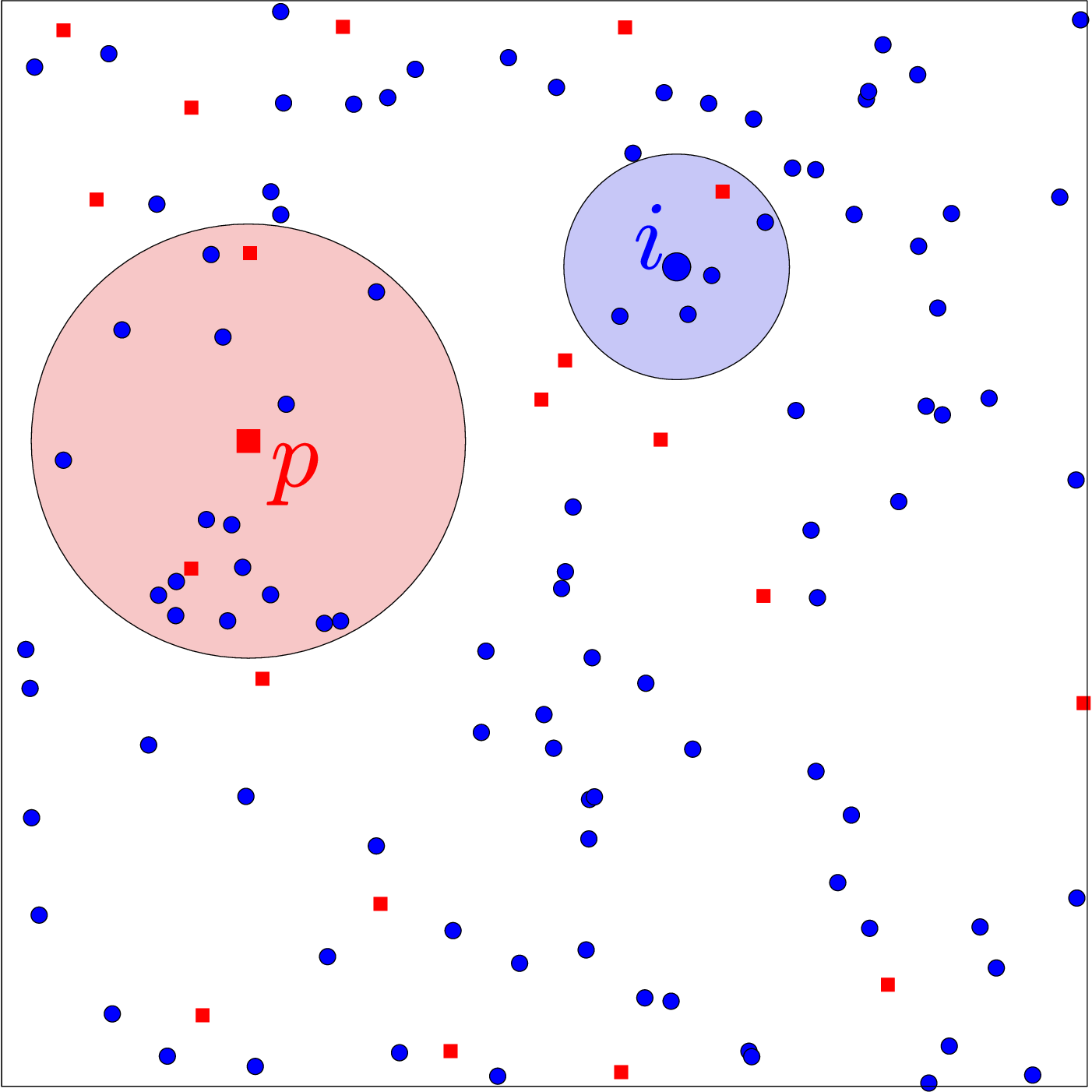}
\caption{Distribution of sources (red) and sensors   (blue). Sensors in the red circle receive signal from source $p$. Sensors in the blue circle can receive data by (but not necessarily send to) sensor $i$.} \label{fig:distribution}
\end{figure}
In our first simulation, the sensors' goal is to collaboratively solve the least square problem
\begin{align}\label{eq:regression}
     \min_{y \in \R^{P}} \ \sum_{\i} \left\| h_i - \mat H_i \col ((y_p)_{p\in \neig{I}{i}) } \right \|^2,
\end{align} 
where $\neig{I}{i}$ is the set of sources positioned less than $r_\textnormal{s}$ away from sensor $i$.
Problem \eqref{eq:regression} is  in the form  \eqref{eq:do1}. We seek a solution via algorithm \eqref{eq:DGD_END} (with fixed communication graph), comparing the performance for two choices of the design graphs:
\begin{itemize}[leftmargin =*]
    \item \emph{Standard}: $\gD{p}$'s are chosen as in \eqref{eq:standard}: with this choice, \eqref{eq:DGD_END} boils down to the standard Push-sum DGD\cite{NedicOlshevsky_Directed_TAC2015}.
    \item \emph{Customized}: each $\gD{p}$ is designed to exploit the sparsity in \eqref{eq:regression}. In particular, we  aim at optimizing the memory allocation for the estimates, by minimizing the number of nodes in $\gD{p} $, provided that $\gD{p} $ must be strongly connected (and \cref{asm:consistency} must be satisfied). Designing such a $\gD{p}$  corresponds to (approximately) solving a Strongly Connected Steiner Subgraph Problem \cite{Charikar:Directed:Steiner:1999} (where $\gD{p}$ is a subgraph of  $\g{C}$)\footnote{We use all the available edges, i.e., $\gD{p} = \g{C}|_{\neigo{E}{p}}$.}.  
\end{itemize}
\begin{figure}
\includegraphics[width=0.95\columnwidth]{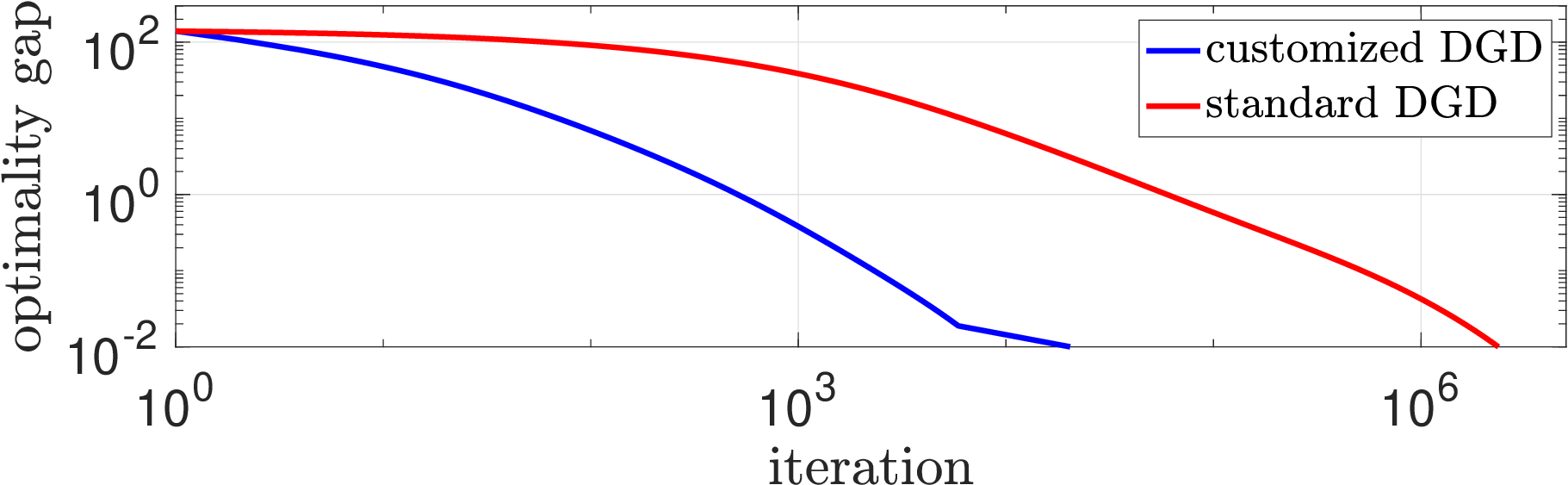}
\\[1em]
\includegraphics[width=0.95\columnwidth]{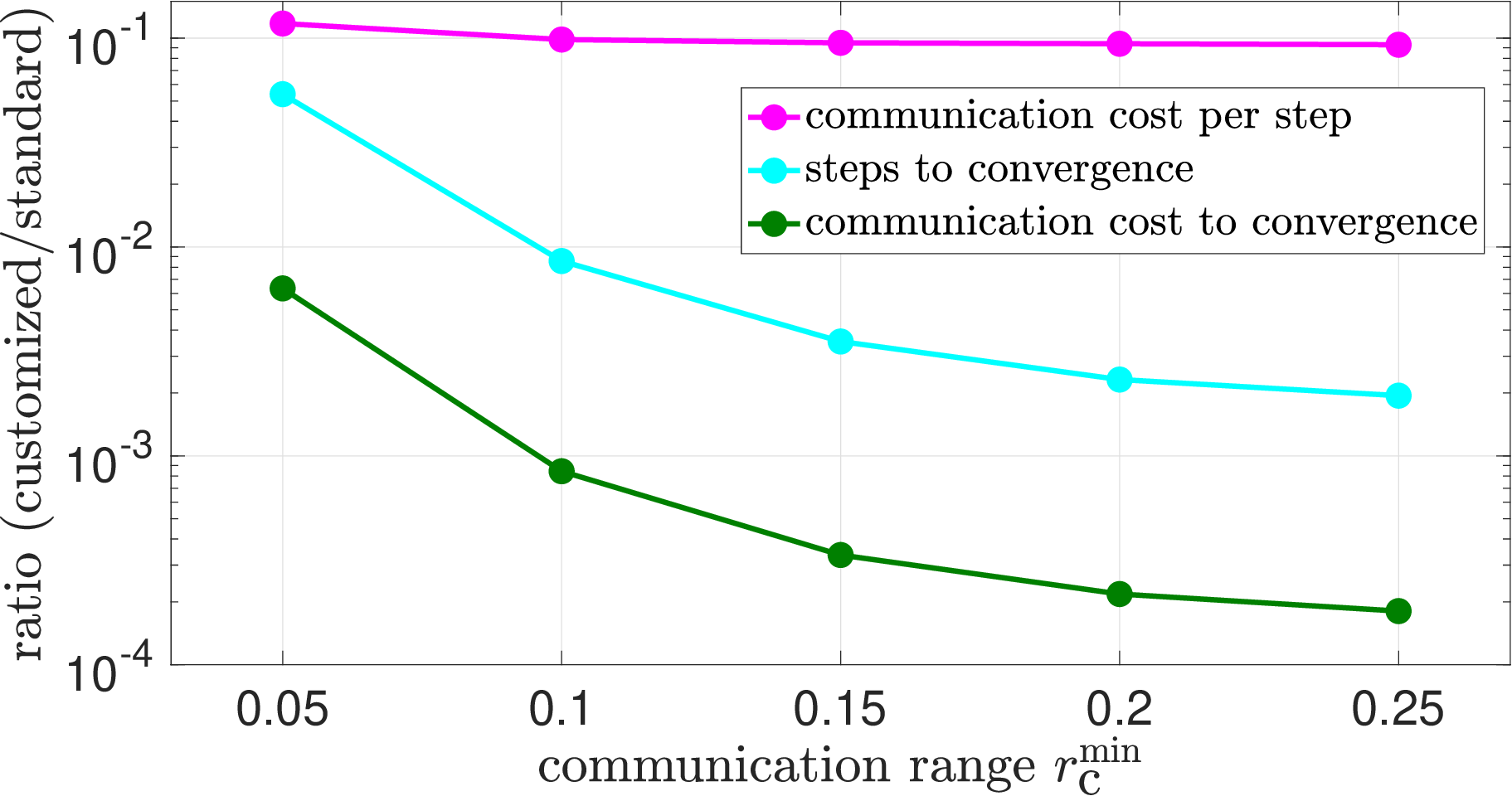}
\caption{Linear regression via algorithm \eqref{eq:DGD_END}, 
	for different values of the minimum sensor communication radius  $r_\textrm{c}^{\min}$ and stopping criterion $\mathfrak V(\hy) \leq 10^{-2}$  (bottom), and the trajectories obtained with $r_\textrm{c}^{\min} = 0.1$ (top). A larger $r_\textrm{c}^{\min}$ induces a denser graph $\g{C}$. } \label{fig:regression_simulation}
\end{figure}

We set $N = 100$, $P=20$,  and randomly  generate sensor/sources positions as in \cref{fig:distribution}.
We choose $r_{\textnormal{s}} = 0.2$, and draw each $r^{i}_{\textnormal{c}}$
uniformly in $[ r_{\textnormal{c}}^{\min},r_{\textnormal{c}}^{\min}+0.1] $.
For all $\i$,  we fix $\n{h_i} = 10$, we generate $\mat H_i$ by first uniformly drawing entries in $[0,1]$ and then normalizing the rows to unitary norm, we  draw each element of $w_i$ from an unbiased normal distribution with variance $0.1$; each signal $\bar y_p$ is uniformly randomly chosen in $[0,1]$; the step size is set as $\gamma^k = k^{-0.51}$ in \eqref{eq:DGD_END}.\footnote{Although the bounded subgradient assumption in \cref{th:DGDABC} fails, boundedness of the sequences generated by \eqref{eq:DGD_END}, and hence convergence, can be established based on coercivity of the cost function.}  The advancement is evaluated via the merit function $\mathfrak V  (\hy)\coloneqq \max \{ \hspace{-0.1em}\| \hspace{-0.1em}\diag( (\frac{1}{N_p}\mat I)_\p)\Piperp{} \hy \| \|\grad{\hy} \fbs( \hy^\star ) \|, \hspace{-0.1em}| \fbs(\Piparallel{}\hy )\hspace{-0.1em} -\hspace{-0.1em} \fbs(\hy^\star) |  \}$, where $\hy^\star = \Ebs{}(y^\star)$ and $y^\star $ solves \eqref{eq:regression}.
\cref{fig:regression_simulation} shows the results for different values of $r_{\textnormal{c}}^{\min}$. For $r_{\textnormal{c}}^{\min} = 0.1$, the customized method is  15 times faster then the standard one. 
Increasing  $r_{\textnormal{c}}^{\min}$  only marginally reduces the  per-iteration communication cost of the customized method. In fact, already for $r_{\textnormal{c}}^{\min} = 0.25$, the graph $\g{C}|_{\neigo{I}{p}}$ is strongly connected for all $\p$, so $\g{E} = \g{I} $ can be chosen (in other terms, each agent only estimates and exchanges the components of $y$ that directly affect its local cost, while it also has to estimate other components for smaller $r_{\textnormal{c}}^{\min}$). 
In this situation, the customized method achieves a reduction of the communication cost (where sending a variable to \emph{all} the neighbors on $\g{C}$ has a cost of $1$, in a broadcast fashion) of over $99.9 \%$.


\subsubsection{LASSO}\label{sec:lasso} 

Next, we assume that only $30 \%$ of the sources emits a signal at a given instant (the vector $\bar y$ is sparse). The sensors collaboratively solve the following problem, regularized to promote sparsity,
\begin{align*}
     \min_{y \in \R^{P}} \ \|y\|_1+ \sum_{\i} \left\| h_i - \mat H_i \col (y_p)_{p\in \neig{I}{i} } \right \|^2,
\end{align*}
where $\|\cdot\|_1$ is the $\ell_1$ norm. 
By defining  $f_i((y_p)_{p \in \neig{I}{i}}) = \| h_i - \mat H_i \col ((y_p)_{p \in \neig{I}{i}}) \|^2+  \sum_{p \in \neig{I}{i}} \frac{1}{|\neigo{I}{p}|}|y_p|$, we retrieve the form \eqref{eq:do1}.
We set $N = 10$, $P=20$, $r_\textnormal{c}^{\min} = 0.1$, $\n{h_i} =1$ for all $i$, generate random positions for the sensors and sources,  and choose the other parameters as above, for both the standard and customized methods. \cref{fig:LASSO} compares the results for different values of $r_\textnormal{s}$. 
For larger $r_\textnormal{s}$, the interference graph $\g{I}$ is denser, and the gap between customized and standard method decreases: in fact, for $r_\textnormal{s} =0.8$  the two algorithms coincide, as $\g{I}$ is complete. Nonetheless, when  $\g{I}$ is sparse, the customized algorithm saves up to $99\%$ of the communication cost. 

In conclusion, while requiring some initial computational  effort to choose the design graphs $\gD{p}$,
the sparsity-aware method results in substantial efficiency
improvement --  especially if the estimation problem is solved
repeatedly, e.g., each time new signals are received from the sources.

\begin{figure}[t] 
\centering
\includegraphics[width=0.95\columnwidth]{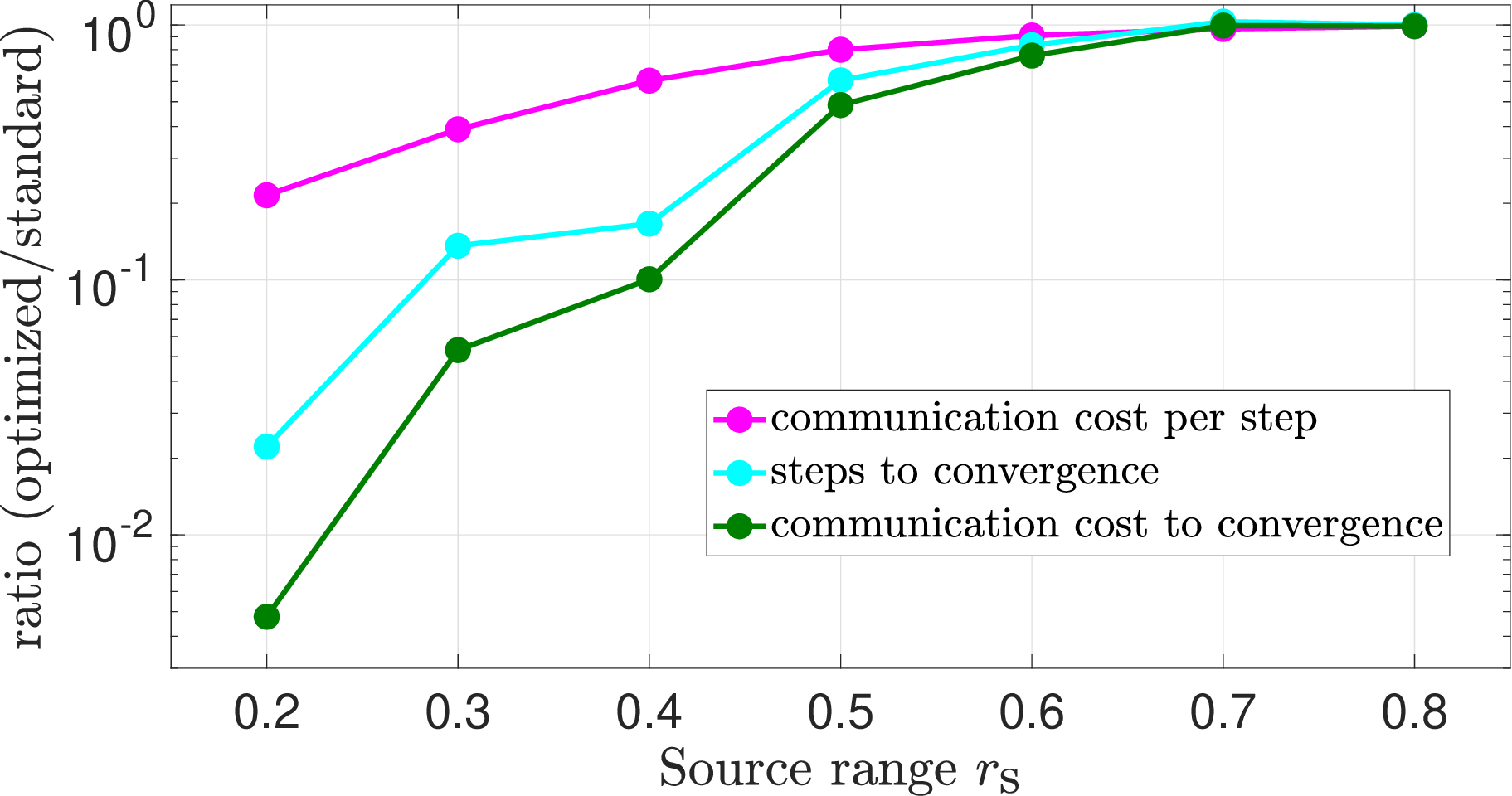}
\caption{LASSO via 
algorithm \eqref{eq:DGD_END},
	and different source ranges $r_{\textrm s}$.}
	\label{fig:LASSO}
\end{figure}


\section{Conclusion}\label{sec:extension}

We have shown that the \gls{END} framework \cite{Bianchi_minG_TCNS_2023} can be applied to a variety of distributed optimization problems, to enhance efficiency by accounting for the intrinsic sparsity in the agents coupling. Besides revisiting dual methods, we derived the \gls{END} (i.e., sparsity aware) version of the very general ABC method and of  Push-sum DGD; and we showed how to efficiently tackle constraint-coupled problems with sparse constraints, even over directed graphs. Our simulations show that \gls{END} algorithms can substantially reduce the computational and communication cost, while introducing little complication in the convergence proof with respect to their sparsity-unaware counterparts. 

As sparsity-aware \gls{END} algorithms require some initial design effort for the allocation of the estimates, their use is particularly recommended for problems with special structure \cite[App. A.4]{Bianchi_minG_TCNS_2023}, or repeated/time-varying problems like distributed estimation and MPC. 
Future work should focus on computationally efficient and
distributed methods to perform the allocation of the estimates online, thus avoiding the need for any a priori
design.

\begin{appendices}
\subsection{Proof of \cref{th:ENDABC}}\label{app:GNEproof}
We adapt the proof of \cite[Th.~24]{Scutari_Unified_TSP2021}. We note that $\hz^0=\0_{\n{\hy}} \in \range(\mat B)$; by the conditions \ref{C1:th:ENDABC} and \ref{C4:th:ENDABC}, the update in \eqref{eq:ENDABC}, and an induction argument, we have  $\hy^{k},\hz^{k} \in \range(\mat B)$, for all $k\geq 1$. Hence, we  rewrite \eqref{eq:ENDABC} as
\begin{subequations}
	\label{eq:unraveled_ABCD}
	\begin{align}
	\hy^k & = \mat B\hyu^{k}, \quad \hz^{k}=\gamma \mat B\hzu^{k}
	\\
	\hyu^{k+1} & = \mat D \hy^{k} -\gamma (\grad{\hy}\fbs(\hy^{k}) +\hzu^{k} )
	\\
	\hzu^{k+1} & = \hzu^{k}+\textstyle \frac{1}{\gamma} \mat  C\hyu^{k+1}
	\end{align}
\end{subequations}
for all $k\geq 1$. Let $\Phi(\hy,\hz) \coloneqq \fbs(\hy)+\langle \hy,\hz \rangle$; the form in \eqref{eq:unraveled_ABCD} can be exploited to prove the following lemma.
\begin{lemma}\label{lem:23Scutari}
	Let $(\hy^k,\hyu^{k},\hzu^{k})$ be a sequence generated by \eqref{eq:unraveled_ABCD}.
	Then, for all  $\hy \in \Ebs{}$, $\hz \in  \Ebsperp{}$, it holds that:
	\begin{equation*}
	\Phi(\hyavg^{k+1},\hz)  - \Phi(\hy,\hz) \leq \textstyle \frac{1}{2k} h (\hy,\hz),
	\end{equation*}
	where
	$	h (\hy,\hz)\coloneqq \frac{1}{\gamma} \| \hy^0 -\hy\|^2_{\mat D}+ \gamma \frac{ \| \mat B-\Piparallel{}  \|}{\underline{\uplambda}}\| \hz\|^2
	$
	and
	$
	\underline{\uplambda}\coloneqq\min \{  (\uplambda_2 (\mat  C_p))_{\p} \}
	$.
	\hfill $\square$
\end{lemma}
\begin{proof}
	The proof is analogous to that of \cite[Lemma~23]{Scutari_Unified_TSP2021}, and omitted here. Note  that \cite{Scutari_Unified_TSP2021} uses a matrix notation (i.e., $\hy \in \R^{I \times n}$), while we need a stacked notation (as the vectors $(\hy_i)_{\i}$ are not homogeneous in size). Nonetheless, \eqref{eq:unraveled_ABCD} matches \cite[Eq.~(33)]{Scutari_Unified_TSP2021}, which allows us to repeat all the steps in  \cite[Lem.~23]{Scutari_Unified_TSP2021} (with the only precaution of replacing
	$J$, $\textrm{span} (\1_m)$, $\uplambda_2 (\mat C)$ in \cite{Scutari_Unified_TSP2021} 	with $\Piparallel{}$, $\Ebs{} $, $\underline{\uplambda}$).
\end{proof}
For all $\hz \in \Ebsperp{}$ (so that $\langle \hz, \hy^\star \rangle = 0  $), setting $\hy=\hy^\star$ in \cref{lem:23Scutari}, together with the definition of $\Phi$, yields
$\fbs(\hyavg^k) - \fbs(\hystar) + \langle \hyavg^k, \hz \rangle  \leq  \frac{1}{2k}  h ( \hystar, \hz )$.
Further choosing $\hz = 2 \frac{\Piperp{} \hyavg^k }{\| \Piperp{} \hyavg^k \| } \| \hzstar \|$, with $\hzstar\coloneqq-\grad{\hy} \fbs(\hystar)$, leads to
\begin{equation}\label{eq:Scutari36plus}
\fbs(\hyavg^k) - \fbs(\hystar) + 2 \| \hzstar \| \left\| \Piperp{} \hyavg^{k} \right \|    \leq  \textstyle \frac{1}{2k}  h ( \hystar, 2 \hzstar ).
\end{equation}
By convexity and since $\hz^\star \in \Ebsperp{}$ (by  optimality conditions), it holds that $f(\hyavg^k)-f(\hystar) \geq -\langle \hyavg^k-\hystar,\hzstar \rangle = -\langle \Piperp{} \hyavg^k,\hzstar \rangle \geq - \| \Piperp{} \hyavg^k \| \| \hzstar \|$; the latter inequality and  \eqref{eq:Scutari36plus} imply $\mathfrak{M} (\hyavg^k) \leq \frac{1}{2k} h ( \hystar, 2 \hzstar )$. \hfill$\blacksquare$
\subsection{Proof of \cref{th:DGDABC}}
Note that, for each $\p$, \eqref{eq:DGD_END} is the standard perturbed push-sum protocol \cite[Eq.~(4)]{NedicOlshevsky_Directed_TAC2015}, with perturbation term $-\gamma^k \hg^{k+1}_{i,p}$. Therefore, since $\hg^{k+1}_{i,p}$ is uniformly bounded by assumption and by the choice of  $(\gamma^k)_{\k}$, we can apply  \cite[Lem.~1]{NedicOlshevsky_Directed_TAC2015} to infer that, for all $\i$, $p\in \neig{E}{i}$
\begin{align}
	\lim_{k\to\infty}  \| \hy_{i,p}^k- \hzbar^k_p  \| =0,  \label{eq:limconsensus}
	\\
	\textstyle \sum_{k=0}^{\infty} \gamma^k \| \hy_{i,p}^{k}-  \hzbar^k_p \| = 0,  \label{eq:sumconsensus}
\end{align}
where $\hzbar_p^k \coloneqq\frac{1}{N_p} \sum_{i \in \neigo{E}{p}} \hz_{i,p}^k \in \R^{\n{y_p}}$, for all $\k$.
Let us also define $\hzbar^k\coloneqq \col( ( \hzbar^k_p )_{ \p } ) \in \R^{\n{y}}$. By \eqref{eq:DGD_END} and  \cref{asm:columnstoch}(ii), it follows that
\begin{align}\label{eq:averaged_process}
	\hzbar^{k+1}_p = \hzbar^{k}_p - \textstyle \gamma^k \frac{1}{N_p}   \sum_{i \in \neigo{E}{p}}  \hg_{i,p} ^{k+1}.
\end{align}
We next show that $	\lim_{k\to\infty}   \hzbar^k = y^\star \in \mc{Y}^\star$; then, the theorem follows by \eqref{eq:limconsensus}.
The main complication with respect to the proof of \cite[Th.~1]{NedicOlshevsky_Directed_TAC2015} is that we need a modification of  \cite[Lem.~8]{NedicOlshevsky_Directed_TAC2015} to cope with the non-homogeneity of the estimates.
\begin{lemma}\label{lem:ineq_DGD}
	For all $y^\star \in \mc{Y}^\star$, for all $\k$, it holds that
	\begin{align*}
		\begin{aligned}[b]
			\| \hzbar^{k+1} -\y^\star \|^2_{\mat{D}} & \leq   \| \hzbar^{k} -\y^\star \|^2_{\mat D} -2\gamma^k({f}( \hzbar^{k}) -{f}(y^\star) )
			\\
			& \hphantom{ {}={}} +  4L \gamma^k \sum_{\i} \sum_ {p\in \neigo{E}{i}}   \| \hzbar^{k}_p - \hyt_{i,p}^{k+1} \|
			\\
			& \hphantom{ {}={}} + (\gamma^k)^2 N L^2,
		\end{aligned}
	\end{align*}
	where $\mat D\coloneqq \diag (( N_p \id_{n_p})_{\p} )$. 	\hfill $\square$
\end{lemma}
\begin{proof} By \eqref{eq:averaged_process}, we have
	\begin{align}\label{eq:usefulstep}  
		\begin{aligned}[b]
			\| \hzbar^{k+1} -\y^\star \|^2_{\mat D} & =  \| \hzbar^{k} -\y^\star \|^2_{\mat D}
			\\
			& \hphantom{ {}={} }
			- 2 \gamma^{k} \sum_{\p}  \left\langle \hzbar^{k}_p -\y_p^\star , \textstyle  \sum_{i \in \neigo{E}{p}}  \hg_{i,p} ^{k+1} \right\rangle
			\\
			& \hphantom{ {}={} }
			+ (\gamma^k)^2  \sum_{\p}    \textstyle   \frac{1}{N_p}  \left\| \sum_{i \in \neigo{E}{p}}  \hg_{i,p} ^{k+1}  \right\|^2\!.
		\end{aligned}\hspace{-1em}
	\end{align}
	The third addend on the right-hand side of \eqref{eq:usefulstep} is bounded above by $(\gamma^k)^2 NL^2$. For the second addend, we have
	\begin{align*}
		\begin{aligned}
			& \hphantom{{}={}}	\sum_{\p}  \left\langle \hzbar^{k}_p -\y_p^\star , \textstyle  \sum_{i \in 	\neigo{E}{p}}  \hg_{i,p} ^{k+1} \right\rangle
			\\
			& = \sum_{\i} \sum_{p \in \neigo{E}{i}}  \left \langle (\hzbar^{k}_p- \hy_{i,p}^{k+1})+(\hy_{i,p}^{k+1}- \y^\star_p),    \hg_{i,p}^{k+1} \right\rangle
			\\
			& \overset{(a)}{ \geq}   \sum_{\i}   -L \textstyle  \| \col( (\hzbar^{k}_p)_{p\in \neigo{E}{i}})- \hyt_{i}^{k+1} \|   + f_i(\hyt_{i}^{k+1})- f_i(y^\star)
			\\
			&  \overset{(b)}{ \geq}   \sum_{\i}   -2L \textstyle  \| \col( (\hzbar^{k}_p)_{p\in \neigo{E}{i}})- \hyt_{i}^{k+1} \|   + f_i(\hzbar^{k+1})- f_i(y^\star),
		\end{aligned}
	\end{align*}
	where in (a) we used that $ \hg_{i}^{k+1}\in \subd{\, \hyt_{i}} f_i(\hyt_{i}^{k+1})$ and convexity of $f_i$, and (b) follows by adding and subtracting (inside the sum) $f_i ( (\hzbar^{k+1}_p )_{p\in\neig{E}{i}}) = f_i(\hzbar^{k+1}) $ and by $L$-Lipschitz continuity of $f_i$. The result follows by substituting the bound back into \eqref{eq:usefulstep}.
\end{proof}
We finally note that, due to \eqref{eq:sumconsensus} and  the choice of $(\gamma_k)_\k$,   the inequality in \cref{lem:ineq_DGD} satisfies all the conditions of \cite[Lem.~7]{NedicOlshevsky_Directed_TAC2015}, in the norm $\|\cdot \|_{\mat D}$;  hence we can conclude that
$\hzbar^k \rightarrow y^\star$, for some $\y^\star \in \mc{Y}^\star$.  \hfill $\blacksquare$

\bibliographystyle{IEEEtran}
\bibliography{library}

\end{appendices}

\end{document}